\documentclass{amsproc}
\usepackage{amsfonts}

\newtheorem{theorem}{\bf Theorem}[section]
\theoremstyle{plain}

\newtheorem{corollary}{\bf Corollary}[section]

\newtheorem{definition}{\bf Definition}[section]
\newtheorem{example}{\bf Example}[section]

\newtheorem{lemma}{\bf Lemma}[section]

\newtheorem{proposition}{\bf Proposition}[section]

\newtheorem{rem}{\bf Remark}[section]
\numberwithin{equation}{section}

\begin{document}
\title[Warped product pseudo-slant submanifolds of nearly Kenmotsu $f$%
-manifolds]{Generalized inequalities of warped product submanifolds of nearly Kenmotsu $f$-manifolds}
\author{Yavuz Selim Balkan, Aliya Naaz Siddiqui and Akram Ali}
\address[ Yavuz Selim Balkan ]{ Department\ of\ Mathematics, Faculty\ of\
Art\ and\ Sciences, Duzce University, 81620, Duzce/TURKEY}
\email{y.selimbalkan@gmail.com}
\address[ Aliya Naaz Siddiqui ]{ Department of Mathematics, Jamia Millia
Islamia, New Delhi-110 025, India,}
\email{aliyanaazsiddiqui9@gmail.com}
\address[ Akram Ali ]{ Department of mathematics, King Khalid University,
Abha Saudi Arabia,}
\email{akramali133@gmail.com}

\begin{abstract}
In the present paper, we discuss the non-trivial warped product pseudo slant
submanifolds of type $M_{\bot }\times _{f}M_{\theta }$ and $M_{\theta
}\times _{f}M_{\bot }~$of nearly Kenmotsu $f$-manifold $\overline{M}$.
Firstly, we get some basic properties of these type warped product
submanifolds. Then, we establish the general sharp inequalities for squared
norm of second fundamental form for mixed totally geodesic warped product
pseudo slant submanifolds of both cases, in terms of the warping function and the slant
angle. Also the equality cases are verified. We show that some previous results are trivial from our results.
\end{abstract}

\maketitle

\section{Introduction}
The warped product which is a natural generalization of Riemannian product
was introduced to construct the manifolds with negative curvature by Bishop
and O'Neill in 1969 \cite{Bishop}. Then Kenmotsu introduced a remarkable
class of almost contact manifolds with negative curvature $-1~$by using
warped product \cite{Kenmotsu}. In 1978, Bejancu studied a special class of K%
\"{a}hler manifolds and defined the $CR$-submanifolds \cite{Bejancu}. Later on, in \cite%
{Chen}
Chen studied $CR$-submanifolds and introduced $CR$-warped product
submanifold in a K\"{a}hler manifold by using these two notions. He established a general inequality for the second fundamental form
in terms of warping functions for an arbitrary $CR$-warped product in an
arbitrary K\"{a}hler manifold. After that many authors derived the geometric inequalities of warped product
submanifolds in different ambient spaces (\cite{Ali1}-\cite{Ali}, \cite%
{Uddin}, \cite{Uddin1}, \cite{Uddin2}). Recently, \c{S}ahin \cite{Sahin}
constructed a general inequality for warped product pseudo slant
isometrically immersed in a K\"{a}hler manifold.

On the other hand, Yano defined and studied the $\left( 2n+s\right) $%
-dimensional globally framed metric $f$-manifold which is a natural
generalization complex manifolds and contact manifolds \cite{Yano}. Then in
1964, Ishihara and Yano investigate the integrability of the structures
defined on these manifolds \cite{Ishihara1}. Blair introduced three classes
of globally framed metric $f$-manifold as $K$-manifolds,~$S$-manifolds and $%
C $-manifolds \cite{Blair}. Moreover Falcitelli and Pastore defined almost
Kenmotsu $f$-manifold which is a generalization of an almost Kenmotsu
manifold \cite{Falcitelli}. \"{O}zt\"{u}rk et. al generalized the almost $C$%
-manifolds and almost Kenmotsu $f$-manifold and they introduced almost $%
\alpha $-cosymplectic $f$-manifolds \cite{Ozturk}. Recently, Balkan studied
a globally framed version of nearly Kenmotsu manifolds and obtained the
fundamental properties of these type manifolds \cite{Balkan1}.

The main objective of this paper to consider nearly Kenmotsu $f$%
-manifolds and compute some geometric sharp inequalities of non-trivial
warped product pseudo slant submanifolds. We prove the existence of the  warped product pseudo slant submanifolds in
nearly Kenmotsu $f$-manifolds by
constructing some examples. It well known that the warped product pseudo slant
submanifolds are natural extensions of $CR$-warped product submanifold with some geometric condition.

\section{Preliminaries}

Let $\overline{M}$ be $\left( 2n+s\right) $-dimensional manifold and $%
\varphi $ is a non-null $\left( 1,~1\right) $ tensor field on $M$. If $%
\varphi $ satisfies
\begin{equation}
\varphi ^{3}+\varphi =0,  \label{1}
\end{equation}%
then $\varphi $ is called an $f$-structure and $\overline{M}$ is called an $%
f $-manifold \cite{Yano}. If $rank\varphi =2n$ or $rank\varphi =2n+1$, i.e.,
$s=0$ or $s=1,~$then $\varphi $ is called an almost complex structure or an
almost contact structure, respectively \cite{Goldberg1}. On the other hand, $%
rank\varphi $ is always constant \cite{Stong}.

On an $f$-manifold $\overline{M}$, $P_{1}$ and $P_{2}$ operators are defined
by
\begin{equation}
\begin{array}{cc}
P_{1}=-\varphi ^{2}, & P_{2}=\varphi ^{2}+I,%
\end{array}
\label{2}
\end{equation}%
which satisfy
\begin{equation}
\begin{array}{ccc}
P_{1}+P_{2}=I, & P_{1}^{2}=P_{1}, & P_{2}^{2}=P_{2}, \\
\varphi P_{1}=P_{1}\varphi =\varphi , & P_{2}\varphi =\varphi P_{2}=0. &
\end{array}
\label{3}
\end{equation}%
These properties show that $P_{1}$ and $P_{2}$ are complement projection
operators. There are $D$~and $D^{\bot }$ distributions with respect to $%
P_{1} $ and $P_{2}$ operators, respectively \cite{Yano1}. Moreover $\dim
\left( D\right) =2m$ and $\dim \left( D^{\bot }\right) =s.$

Let $\overline{M}$ be $\left( 2m+s\right) $-dimensional $f$-manifold and $%
\varphi $ is a $\left( 1,~1\right) $ tensor field, $\xi _{i}$ is vector
field and $\eta ^{i}$ is $1$-form for each $1\leq i\leq s$ on $M,$
respectively. If the following properties are satisfied
\begin{equation}
\eta ^{j}\left( \xi _{i}\right) =\delta _{i}^{j},  \label{4}
\end{equation}%
and
\begin{equation}
\varphi^{2}=-I+\sum\limits_{i=1}^{s}\eta ^{i}\otimes \xi _{i},  \label{5}
\end{equation}%
then $\left( \varphi ,~\xi _{i},~\eta ^{i}\right) $ is called globally
framed $f$-structure or simply framed $f$-structure and $\overline{M}$ is
called globally framed $f$-manifold or simply framed $f$-manifold \cite%
{Nakagawa}. For a framed $f$-manifold $\overline{M},$\ the following
properties hold \cite{Nakagawa}:%
\begin{equation}
\varphi \xi _{i}=0,  \label{6}
\end{equation}%
\begin{equation}
\eta ^{i}\circ \varphi =0.  \label{7}
\end{equation}%
On a framed $f$-manifold $\overline{M}$ if there exists a Riemannian metric
which satisfies
\begin{equation}
\eta ^{i}\left( X\right) =g\left( X,~\xi _{i}\right) ,  \label{8}
\end{equation}%
and
\begin{equation}
g\left( \varphi X,~\varphi Y\right) =g\left( X,~Y\right)
-\sum\limits_{i=1}^{s}\eta ^{i}\left( X\right) \eta ^{i}\left( Y\right) ,
\label{0}
\end{equation}%
for all vector fields $X,~Y$ on $\overline{M},$ then $\overline{M}~$is
called a framed metric $f$-manifold \cite{Goldberg}. On a framed metric $f$%
-manifold, the fundamental $2$-form $\Phi $ is defined by
\begin{equation}
\Phi \left( X,~Y\right) =g\left( X,~\varphi Y\right) ,  \label{9}
\end{equation}%
for all vector fields $X,~Y$ on $\overline{M}$ \cite{Goldberg}. For a framed
metric $f$-manifold, if the following holds%
\begin{equation}
N_{\varphi }+2\sum\limits_{i=1}^{s}d\eta ^{i}\otimes \xi _{i}=0,  \label{10}
\end{equation}%
then $\overline{M}$ is called normal framed metric $f$-manifold, where $%
N_{\varphi }$ denotes the Nijenhuis torsion tensor of $\varphi $ \cite%
{Ishihara}.

A globally framed metric $f$-manifold $\overline{M}$ is called Kenmotsu $f$%
-manifold if it satisfies
\begin{equation}
\left( \overline{\nabla }_{X}\varphi \right) Y=\sum_{k=1}^{s}\left\{ g\left(
\varphi X,~Y\right) \xi _{k}-\eta ^{k}\left( Y\right) \varphi X\right\} ,
\label{00}
\end{equation}%
for all vector fields $X,~Y$ on $\overline{M}$ \cite{Ozturk}. Furthermore,
if a globally framed metric $f$-manifold $\overline{M}$ satisfies
\begin{equation}
\left( \overline{\nabla }_{X}\varphi \right) Y+\left( \overline{\nabla }%
_{Y}\varphi \right) X=-\sum_{k=1}^{s}\left\{ \eta ^{k}\left( X\right)
\varphi Y+\eta ^{k}\left( Y\right) \varphi X\right\}  \label{11}
\end{equation}%
then it is called a nearly Kenmotsu $f$-manifold. It is easily seen that
every Kenmotsu $f$-manifold is a nearly Kenmotsu $f$-manifold, but the
converse is not true. When a nearly Kenmotsu $f$-manifold $\overline{M}$ is
normal, it turns to a Kenmotsu $f$-manifold \cite{Balkan1}. On a nearly
Kenmotsu $f$-manifold $\overline{M},$ the following identities hold:
\begin{equation}
\overline{\nabla }_{X}\xi _{i}=-\varphi ^{2}X,  \label{000}
\end{equation}%
\begin{equation}
R\left( \xi _{i},~X\right) Y=\sum_{k=1}^{s}\left\{ -g\left( X,~Y\right) \xi
_{k}+\eta ^{k}\left( Y\right) X\right\} ,  \label{12}
\end{equation}%
\begin{equation}
R\left( X,~Y\right) \xi _{i}=\sum_{k=1}^{s}\left\{ \eta ^{k}\left( X\right)
Y-\eta ^{k}\left( Y\right) X\right\} ,  \label{13}
\end{equation}%
\begin{equation}
S\left( \varphi X,~\varphi Y\right) =S\left( X,~Y\right) +\left(
2n+s-1\right) \sum_{k=1}^{s}\eta ^{k}\left( X\right) \eta ^{k}\left(
Y\right) ,  \label{14}
\end{equation}%
\begin{equation}
\left( \overline{\nabla }_{X}\eta ^{i}\right) Y=g\left( X,~Y\right)
-\sum_{k=1}^{s}\eta ^{k}\left( X\right) \eta ^{k}\left( Y\right) ,
\label{15}
\end{equation}%
\begin{equation}
\sum_{k=1}^{s}\eta ^{k}\left( R\left( X,~Y\right) Z\right)
=\sum_{k=1}^{s}\left\{ g\left( X,~Z\right) \eta ^{k}\left( Y\right) -g\left(
Y,~Z\right) \eta ^{k}\left( X\right) \right\} ,  \label{16}
\end{equation}%
for any vector fields $X,~Y$ on $\overline{M}$ \cite{Balkan1}. Now we give the following example for nearly Kenmotsu $f-$manifolds.

\begin{example}
	Let us consider a $14$-dimensional manifold
	
	\begin{equation*}
		\overline{M}=\left\{ \left( x_{1},~\ldots ,~x_{6},~y_{1},~\ldots
		,~y_{6},~z_{1},~z_{2}\right) \in
		\mathbb{R}
		^{14}:~z_{1},~z_{2}\neq 0\right\},
	\end{equation*}%
	where $\left( x_{1},~\ldots ,~x_{6},~y_{1},~\ldots
	,~y_{6},~z_{1},~z_{2}\right) $ are standard coordinates in $%
	\mathbb{R}
	^{14}.$ We set the vector fields
	\begin{equation*}
		\begin{array}{ccc}
			e_{1}=e^{-\left( z_{1}+z_{2}\right) }\dfrac{\partial }{\partial x_{1}},~ &
			e_{2}=e^{-\left( z_{1}+z_{2}\right) }\dfrac{\partial }{\partial x_{2}},~ &
			e_{3}=e^{-\left( z_{1}+z_{2}\right) }\dfrac{\partial }{\partial x_{3}}%
			,~\bigskip \\
			e_{4}=e^{-\left( z_{1}+z_{2}\right) }\dfrac{\partial }{\partial x_{4}},~ &
			e_{5}=e^{-\left( z_{1}+z_{2}\right) }\dfrac{\partial }{\partial x_{5}},~ &
			e_{6}=e^{-\left( z_{1}+z_{2}\right) }\dfrac{\partial }{\partial x_{6}}%
			,~\bigskip \\
			e_{7}=e^{-\left( z_{1}+z_{2}\right) }\dfrac{\partial }{\partial y_{1}},~ &
			e_{8}=e^{-\left( z_{1}+z_{2}\right) }\dfrac{\partial }{\partial y_{2}},~ &
			e_{9}=e^{-\left( z_{1}+z_{2}\right) }\dfrac{\partial }{\partial y_{3}}%
			,~\bigskip \\
			e_{10}=e^{-\left( z_{1}+z_{2}\right) }\dfrac{\partial }{\partial y_{4}},~ &
			e_{11}=e^{-\left( z_{1}+z_{2}\right) }\dfrac{\partial }{\partial y_{5}},~ &
			e_{12}=e^{-\left( z_{1}+z_{2}\right) }\dfrac{\partial }{\partial y_{6}}%
			,~\bigskip \\
			e_{13}=\dfrac{\partial }{\partial z_{1}},~ & e_{14}=\dfrac{\partial }{%
				\partial z_{1}} ,&
		\end{array}%
	\end{equation*}%
	which are linearly independent at any point of $\overline{M}.$ Let $g$ be
	the Riemannian metric defined by
	\begin{equation*}
		g=e^{2\left( z_{1}+z_{2}\right) }\left( \sum_{i=1}^{6}\left( dx_{i}\otimes
		dx_{i}+dy_{i}\otimes dy_{i}\right) +\eta _{1}\otimes \eta _{1}+\eta
		_{2}\otimes \eta _{2}\right)
	\end{equation*}%
	where $\eta _{1}$ and $\eta _{2}$ are $1$-forms defined by $\eta _{1}\left(
	X\right) =g\left( X,~e_{13}\right) $ and $\eta _{2}\left( X\right) =g\left(
	X,~e_{14}\right) $, respectively. Thus $\left( e_{1},~\ldots ,~e_{14}\right)
	$ is an orthornormal basis of $\overline{M}$. Thus we define the $\left(
	1,1\right) $-tensor field $\varphi $ as follows:
	\begin{equation*}
		\varphi \left( \sum_{i=1}^{6}\left( x_{i}\dfrac{\partial }{\partial x_{i}}%
		+y_{i}\dfrac{\partial }{\partial y_{i}}\right) +z_{1}\dfrac{\partial }{%
			\partial z_{1}}+z_{2}\dfrac{\partial }{\partial z_{2}}\right) =\varphi
		\left( \sum_{i=1}^{6}\left( x_{i}\dfrac{\partial }{\partial y_{i}}-y_{i}%
		\dfrac{\partial }{\partial x_{i}}\right) \right) .
	\end{equation*}%
	Then we have
	\begin{equation*}
		\begin{array}{ccc}
			\varphi e_{1}=e_{7},~ & \varphi e_{2}=e_{8},~ & \varphi e_{3}=e_{9}\bigskip
			\\
			\varphi e_{4}=e_{10} & \varphi e_{5}=e_{11} & \varphi e_{6}=e_{12}\bigskip
			\\
			\varphi e_{7}=-e_{1},~ & \varphi e_{8}=-e_{2},~ & \varphi
			e_{9}=-e_{3}\bigskip \\
			\varphi e_{10}=-e_{4} & \varphi e_{11}=-e_{5} & \varphi e_{12}=-e_{6}\bigskip
			\\
			\varphi e_{13}=0 & \varphi e_{14}=0. &
		\end{array}%
	\end{equation*}%
	From the linearity of $g$ and $\varphi $, it yields that
	\begin{eqnarray*}
		&&%
		\begin{array}{ccc}
			\eta _{1}\left( \xi _{1}\right) =1,~ & \eta _{2}\left( \xi _{2}\right) =1,~
			& \varphi ^{2}X=-X+\eta _{1}\left( X\right) \xi _{1}+\eta _{2}\left(
			X\right) \xi _{2}%
		\end{array}
		\\
		&&%
		\begin{array}{c}
			g\left( \varphi X,~\varphi Y\right) =g\left( X,~Y\right) -\eta _{1}\left(
			X\right) \eta _{1}\left( X\right) -\eta _{2}\left( X\right) \eta _{2}\left(
			X\right) .%
		\end{array}%
	\end{eqnarray*}%
	Hence $\left( \varphi ,~\xi _{i},~\eta ^{i},~\right) $ defines a globally
	framed metric $f$-structure on $\overline{M}.$ On the other hand, by the
	virtue of definition of $1$-forms $\eta _{1}$ and $\eta _{2}$ it is said
	that they are closed. In other words $d\eta _{1}=0$ and $d\eta _{2}=0.$
	Moreover, we get the fundamental $2$-form $\Phi $ on $\overline{M}~$as in
	the follwing form
	\begin{eqnarray*}
		\Phi \left( \sum_{i=1}^{6}\dfrac{\partial }{\partial x_{i}},~\sum_{i=1}^{6}%
		\dfrac{\partial }{\partial y_{i}}\right)=&g\left( \sum_{i=1}^{6}\dfrac{\partial }{\partial x_{i}}%
		,~\sum_{i=1}^{6}\varphi \dfrac{\partial }{\partial y_{i}}\right) \\
		=&g\left( \sum_{i=1}^{6}\dfrac{\partial }{\partial x_{i}},~-\sum_{i=1}^{6}%
		\dfrac{\partial }{\partial x_{i}}\right) \\
		=&e^{-2\left( z_{1}+z_{2}\right) }.
	\end{eqnarray*}%
	Then we have $\Phi =-e^{2\left( z_{1}+z_{2}\right) }\sum_{i=1}^{6}\left(
	dx_{i}\wedge dy_{i}\right) .$ From the exterior derivative of the
	fundamental $2$-form $\Phi ,$ then we derive
	\begin{equation*}
		d\Phi =-2e^{2\left( z_{1}+z_{2}\right) }\left( dz_{1}+dz_{2}\right) \wedge
		\sum_{i=1}^{6}\left( dx_{i}\wedge dy_{i}\right) =2\left( \eta _{1}+\eta
		_{2}\right) \wedge \Phi .
	\end{equation*}%
	Then the manifold $\overline{M}$ is called an almost Kenmotsu $f$-manifold.
	After some easy calculation, it is easy to prove that seen that  the manifold $\overline{M}$ is normal. Hence by fact
	that every Kenmotsu $f$-manifold is a nearly Kenmotsu $f$-manifold from \cite%
	{Balkan1}, then we arrive at $\overline{M}$ is a nearly Kenmotsu $f$%
	-manifold.
\end{example}

\begin{rem}\label{r21}
From \eqref{11}, it is clear that if $s=0$ then $\overline{M}$ is become nearly Kaehler manifold \cite{Uddin3}. If $s=1$ then the manifold $\overline{M}$ is called nearly Kenmotsu manifold \cite{Ali}.
\end{rem}

Now we recall some basic facts of submanifold from \cite{Chen0}. Let $M$ be
a submanifold immersed in $\overline{M}$. We also denote by $g$ the induced
metric on $M$. Let $TM$ be the Lie algebra of vector fields in $M$ and $%
T^{\bot }M$ the set of all vector fields normal to $M$. Denote by $\nabla $
and $\overline{\nabla }$ the Levi-Civita connections of $M$ and $\overline{M}%
,$ respectively. Then the Gauss and Weingarten formulas are given by
\begin{equation}
\overline{\nabla }_{X}Y=\nabla _{X}Y+h\left( X,~Y\right)  \label{a}
\end{equation}%
and
\begin{equation}
\overline{\nabla }_{X}V=-A_{V}X+\nabla _{X}^{\bot }V  \label{b}
\end{equation}%
respectively, for any vector fields $X,~Y$ on $\overline{M}$ and any $V\in
T^{\bot }M.$ Here, $\nabla ^{\bot }$ is normal connection in the normal
bundle, $h$ is second fundamental form of $M$ and $A_{V}$ is the Weingarten
endomorphism associated with $V$. On the other hand, there is a relation
between $A_{V}$ and $h$ such that
\begin{equation}
g\left( A_{V}X,~Y\right) =g\left( h\left( X,~Y\right) ,~V\right).  \label{c}
\end{equation}
The mean curvature vector $H$ is defined by $H=\dfrac{1}{m}traceh$, where $m$
is the dimension of $M$. $M$ is said to be minimal, totally geodesic and
totally umbilical if $H$ vanishes identically and $h=0$,
\begin{equation}
h\left( X,~Y\right) =g\left( X,~Y\right) H,  \label{d}
\end{equation}%
respectively. Furthermore, the second fundamental form $h$ satisfies
\begin{equation}
\left( \overline{\nabla }_{X}h\right) \left( Y,~Z\right) =\nabla _{X}^{\bot
}h\left( Y,~Z\right) -h\left( \nabla _{X}Y,~Z\right) -h\left( Y,~\nabla
_{X}Z\right) .  \label{e}
\end{equation}

\section{Submanifolds of Globally Framed Metric $f$-manifolds}

In this section, let us recall some basic properties of submanifolds of
globally framed metric $f$-manifolds from \cite{Balkan}.

\begin{definition}
Let $\overline{M}$ be a globally framed metric $f$-manifold and $M$ is a
submanifold of $\overline{M}.$ For any vector field $X$ on $M$ $,$ we can
write
\begin{equation}
\varphi X=TX+NX,  \label{f}
\end{equation}%
where $TX$ and $NX$ are called tangent and normal component of $\varphi X$,
respectively. Similarly, for each $V\in \Gamma \left( T^{\bot }M\right) $,
we have
\begin{equation}
\varphi V=tV+nV.  \label{17}
\end{equation}%
Here, $tV$ is tangent component and $nV$ is normal component of $\varphi V.$
\end{definition}

\begin{corollary}
Let $\overline{M}$ be a globally framed metric $f$-manifold and $M$ is a
submanifold of $\overline{M}.$ Then the following identities hold:
\begin{equation}
\begin{array}{cc}
T^{2}=-I+\sum\limits_{k=1}^{s}\eta ^{k}\otimes \xi _{k}-tN, & NT+nN=0,%
\end{array}
\label{18}
\end{equation}%
\begin{equation}
\begin{array}{cc}
Tt+tn=0, & Nt+n^{2}=-I,%
\end{array}
\label{19}
\end{equation}%
where $I$ denotes the identity transformation.
\end{corollary}

\begin{proposition}
Let $\overline{M}$ be a globally framed metric $f$-manifold and $M$ is a
submanifold of $\overline{M}.$ Then, $T$ and $n$ are skew-symmetric tensor
fields.
\end{proposition}

\begin{proposition}
Let $\overline{M}$ be a globally framed metric $f$-manifold and $M$ is a
submanifold of $\overline{M}.$ Then, for any vector field $X$ on $M$ and $%
V\in \Gamma \left( T^{\bot }M\right) $, we have
\begin{equation}
g\left( NX,~V\right) =-g\left( X,~tV\right) ,  \label{20}
\end{equation}%
which gives the relation between $N$ and $t$.
\end{proposition}

\begin{proposition}
Let $\overline{M}$ be a globally framed metric $f$-manifold and $M$ is a
submanifold of $\overline{M}.$ Then, for any vector fields $X,$ $Y~$on $M$
and $V\in \Gamma \left( T^{\bot }M\right) ,$ the following identities hold:
\begin{equation}
\left( \overline{\nabla }_{X}\varphi \right) Y=\overline{\nabla }_{X}\varphi
Y-\varphi \overline{\nabla }_{X}Y  \label{g}
\end{equation}%
\begin{equation}
\left( \overline{\nabla }_{X}T\right) Y=\overline{\nabla }_{X}TY-T\overline{%
\nabla }_{X}Y,  \label{21}
\end{equation}%
\begin{equation}
\left( \overline{\nabla }_{X}N\right) Y=\overline{\nabla }_{X}^{\bot }NY-N%
\overline{\nabla }_{X}Y,  \label{22}
\end{equation}%
\begin{equation}
\left( \overline{\nabla }_{X}t\right) V=\overline{\nabla }_{X}tV-t\overline{%
\nabla }_{X}^{\bot }V,  \label{23}
\end{equation}%
\begin{equation}
\left( \overline{\nabla }_{X}n\right) V=\overline{\nabla }_{X}^{\bot }nV-n%
\overline{\nabla }_{X}^{\bot }V,  \label{24}
\end{equation}%
\begin{equation}
\left( \overline{\nabla }_{X}T\right) Y+\left( \overline{\nabla }%
_{Y}T\right) X=A_{NX}Y+A_{NY}X+2th\left( X,~Y\right) ,  \label{25}
\end{equation}%
\begin{equation}
\left( \overline{\nabla }_{X}N\right) Y+\left( \overline{\nabla }%
_{Y}N\right) X=2nh\left( X,~Y\right) -h\left( X,~TY\right) -h\left(
Y,~TX\right) ,  \label{26}
\end{equation}%
\begin{equation}
\left( \overline{\nabla }_{X}t\right) V=A_{nV}X-TA_{V}X,  \label{27}
\end{equation}%
\begin{equation}
\left( \overline{\nabla }_{X}n\right) V=-h\left( tV,~X\right) -NA_{V}X,
\label{28}
\end{equation}%
where $h$ is the second fundamental form, $\nabla $ is the Levi-Civita
connection and $A_{V}$ denotes the shape operator corresponding to the
normal vector field $V$.
\end{proposition}

\begin{definition}
Let $\overline{M}$ be a globally framed metric $f$-manifold and $M$ is a
submanifold of $\overline{M}.$ Then the $TM$ tangent bundle of $M$ can be
decomposed as
\begin{equation}
TM=\sum\limits_{k=1}^{s}D_{\theta }\oplus \xi _{k},  \label{29}
\end{equation}%
where for each $1\leq k\leq s$ the $\xi _{k}$ denotes the distributions
spanned by the structure vector fields $\xi _{k}$ and $D_{\theta }$ is
complementary of distributions $\xi _{k}$ in $TM$, known as the slant
distribution on $M.$
\end{definition}

\begin{theorem}
Let $\overline{M}$ be a globally framed metric $f$-manifold and $M$ is a
submanifold of $\overline{M}.$ Then $M$ is a slant submanifold if and only
if there exists a constant $\mu \in \left[ 0,~1\right] $ such that
\begin{equation}
T^{2}=-\mu \left( I-\sum\limits_{k=1}^{s}\eta ^{k}\otimes \xi _{k}\right) .
\label{30}
\end{equation}%
Moreover, if $\theta $ is the slant angle of $M$, then $\mu =\cos ^{2}\theta
.$
\end{theorem}

\begin{corollary}
Let $M$ be a slant submanifold of a globally framed metric $f$-manifold $%
\overline{M}$ with slant angle $\theta .$ Then for any vector fields $X,$ $%
Y~ $on $M,$ we find
\begin{equation}
g\left( TX,~TY\right) =\cos ^{2}\theta \left\{ g\left( X,~Y\right)
-\sum\limits_{k=1}^{s}\eta ^{k}\left( X\right) \eta ^{k}\left( Y\right)
\right\}  \label{31}
\end{equation}%
and
\begin{equation}
g\left( NX,~NY\right) =\sin ^{2}\theta \left\{ g\left( X,~Y\right)
-\sum\limits_{k=1}^{s}\eta ^{k}\left( X\right) \eta ^{k}\left( Y\right)
\right\} .  \label{32}
\end{equation}
\end{corollary}

\begin{definition}
Let $M$ be a submanifold of a globally framed metric $f$-manifold $\overline{%
M}$ and let $M$ be tangent to the structure vector fields $\xi _{k}$ for
each $1\leq k\leq s.$ For each nonzero vector $X$ tangent to $M$ at $p$,
we denote by $0\leq \theta \left( X\right) \leq \dfrac{\pi }{2}$, the angle
between $\varphi X$ and $T_{p}M$, known as the Wirtinger angle of $X$. If
the $\theta \left( X\right) $ is constant, that is, independent of the
choice of $p\in M$ and $X\in T_{p}M-\left\{ \xi _{k}\right\} ,$ for each $%
1\leq k\leq s,$ then $M$ is said to be a slant submanifold and the constant
angle $\theta $ is called slant angle of the slant submanifold
\end{definition}

Here, if $\theta =0,$ $M$ is invariant submanifold and if $\theta =\dfrac{%
\pi }{2}$, then $M$ is an anti-invariant submanifold. A slant submanifold is
proper slant if it is neither invariant nor anti-invariant submanifold.

\begin{definition}
Let $M$ be a submanifold of a a globally framed metric $f$-manifold $%
\overline{M}$ \ We say that $M$ is a pseudo-slant submanifold if there exist
two orthogonal distributions $D_{\theta }$ and $D^{\bot }$ such that

$1)$ The $TM$ tangent bundle of $M$ admits the orthogonal direct
decomposition $TM=D^{\bot }\oplus D_{\theta },$ where for each $1\leq k\leq
s $ $\xi _{k}\in \Gamma \left( D_{\theta }\right) .$

$2)$ The distribution $D^{\bot }$ is anti-invariant i. e., $\varphi \left(
D^{\bot }\right) \subset \left( T^{\bot }M\right) .$

$3)$ The distribution $D_{\theta }$ is slant with angle $\theta \neq \dfrac{%
\pi }{2},$ that is, the angle between $D_{\theta }$ and $\varphi \left(
D_{\theta }\right) $ is a constant.
\end{definition}

A pseudo-slant submanifold of a globally framed metric $f$-manifold is
called mixed totally geodesic if $h\left( X,~Z\right) =0$ for all $X\in
\Gamma \left( D^{\bot }\right) $ and $Z\in \Gamma \left( D_{\theta }\right)
. $ Now let $\left\{ e_{1},~\ldots ,~e_{n}\right\} $ be an orthonormal basis
of the tangent space $TM$ and $e_{r}$belongs to the orthonormal basis $%
\left\{ e_{n+1},~\ldots ,~e_{m}\right\} $ of a normal bundle $T^{\bot }M$,
then we define
\begin{equation}
\begin{array}{ccc}
h_{ij}^{r}=g\left( h\left( e_{i},~e_{j}\right) ,~e_{r}\right) & \text{and} &
\left\Vert h\right\Vert ^{2}=\sum_{i,~j=1}^{n}g\left( h\left(
e_{i},~e_{j}\right) ,~h\left( e_{i},~e_{j}\right) \right) .%
\end{array}
\label{33}
\end{equation}
On the other hand, for a differentiable function $\lambda $ on $M$, we have
\begin{equation}
\left\Vert \nabla \lambda \right\Vert ^{2}=\sum_{i=1}^{n}\left( e_{i}\left(
\lambda \right) \right) ^{2},  \label{34}
\end{equation}%
where the gradient $grad\lambda$ is defined by $g\left( \nabla
\lambda ,~X\right) =X\lambda $, for any vector field $X\in \Gamma \left(
TM\right) .$

\section{Warped Product Pseudo-Slant Submanifolds}

In this section, we investigate some fundamental properties of warped \
product pseudo-slant submanifolds of a nearly Kenmotsu $f$-manifold. First,
we recall the definition of warped product manifolds and provide the useful lemma from \cite{Bishop} which will use in the proof of our main results.

\begin{definition}
Let $\left( M_{1},~g_{M_{1}}\right) ~$and $\left( M_{2},~g_{M_{2}}\right) $
be two Riemannian manifolds with Riemannian metrics $g_{M_{1}}$ and $%
g_{M_{2}}$, respectively and $f$ is a positive differentiable function on $%
M_{1}.$ The warped product $M_{1}\times _{f}M_{2}$ of $M_{1}$ and $M_{2}$ is
the Riemannian manifold $\left( M_{1}\times M_{2},~g\right) $, where
\begin{equation*}
g=g_{M_{1}}\times f^{2}g_{M_{2}}.
\end{equation*}%
More explicitly , if $U$ is tangent to $M=M_{1}\times _{f}M_{2}$ at $\left(
p,~q\right) ,$ then
\begin{equation*}
\left\Vert U\right\Vert ^{2}=\left\Vert d_{\pi _{1}}U\right\Vert
^{2}+f^{2}\left\Vert d_{\pi _{2}}U\right\Vert ^{2}
\end{equation*}%
where for $i=1,~2$, $\pi _{i}$ are the canonical projections of $M_{1}\times
M_{2}$ on $M_{1}$ and $M_{2}$ respectively.
\end{definition}

\begin{lemma}\label{l41}
Let $M=M_{1}\times _{f}M_{2}$ be a warped product manifold. Then we have

$\left( i\right) ~\nabla _{X}Y\in \Gamma \left( TM_{1}\right) ,$

$\left( ii\right) ~\nabla _{Z}X=\nabla _{X}Z=\left( X\ln f\right) Z,$

$\left( iii\right) ~\nabla _{Z}W=\nabla _{Z}^{\circ }W-g\left( Z,~W\right)
\nabla \ln f,$

for all $X,~Y~$on $M_{1}$ and $Z,~W~$on $M_{2},$where $\nabla $ and $\nabla
^{\circ }$ denote the Levi-Civita connections on $M_{1}$ and $M_{2},$%
respectively. Moreover, $\nabla \ln f,$ the gradient of $\ln f$, is defined
by $g\left( \nabla \ln f,~U\right) =U\ln f.$ A warped product manifold $%
M=M_{1}\times _{f}M_{2}$ is trivial if the warping function $f$ is
constant.\ If $M=M_{1}\times _{f}M_{2}$ is a warped product manifold then it
is said to be that $M_{1}$ is totally geodesic and $M_{2}$ is totally
umbilical submanifold of $M$.
\end{lemma}
Motivated from the definition of pseudo-slant submanifolds, we see that there are two types of warped product pseudo-slant submanifolds can be constructed. According to them, we have the following cases
\begin{equation}
(i)\,\,\, M_{\theta}\times _{f}M_{\bot }\,\,\,\,\text{and}\,\,\,(ii)\,\,\,M_{\bot }\times _{f}M_{\theta },
\end{equation}
where we consider the the
structure vector fields $\xi _{k}^{\prime }s$~are tangent to base manifolds in both cases.\\

Let us consider the first case of warped product pseudo-slant submanifold of type  $M_{\theta}\times _{f}M_{\bot }$ in a nearly Kenmotsu $f-$manifold. The following examples ensure the existence of warped product pseudo-slant submanifold as follows
\begin{example}
	It is well-known that $S^{6}$ inherit a nearly K\"{a}hler structure, which
	is not K\"{a}hlerian. Let us consider $\overline{M}=S^{6}\times _{f}R^{2}$
	and let $f\left( t_{1},~t_{2}\right) =e^{t_{1}+t_{2}}$. Then it is said that
	$\overline{M}$ is a nearly Kenmotsu $f$ -manifold which is not Kenmotsu $f$
	-manifold with warping function $f.$ On the other hand, $\dim \overline{M}=8$
	where $s=2.$ Assume that the coordinates of $M$ are as following
	\begin{equation*}
		\left\{ x_{1},~y_{1},~x_{2},~y_{2},~x_{3},~y_{3},~t_{1},~t_{2}\right\}
	\end{equation*}%
	and for $1\leq i,~j\leq 3,~k=1,~2,$ the following equations are satisfied
	\begin{equation*}
		\begin{array}{ccc}
			\varphi \left( \dfrac{\partial }{\partial x_{i}}\right) =-\dfrac{\partial }{%
				\partial y_{i}},~ & \varphi \left( \dfrac{\partial }{\partial y_{i}}\right) =%
			\dfrac{\partial }{\partial x_{i}},~ & \varphi \left( \dfrac{\partial }{%
				\partial t_{k}}\right) =0.%
		\end{array}%
	\end{equation*}
	
	Now, we consider a submanifold of $\overline{M}$ defined by the following
	immersion $\chi $
	\begin{equation*}
		\chi \left( u_{1},~u_{2},~u_{3},~t_{1},~t_{2}\right) =\left( u_{3}\sin
		u_{1},~u_{2}\sin u_{1},~u_{3}\text{-}u_{2},~u_{3}\text{+}u_{2},~u_{3}\cos
		u_{1},~u_{2}\cos u_{1},~t_{1},~t_{2}\right) .
	\end{equation*}%
	Then the tangent space of $M$ spanned by
	\begin{equation*}
		W_{1}=\cos u_{1}\dfrac{\partial }{\partial x_{1}}+\dfrac{\partial }{\partial
			x_{2}}+\dfrac{\partial }{\partial y_{2}}+\sin u_{1}\dfrac{\partial }{%
			\partial x_{3}},
	\end{equation*}%
	\begin{equation*}
		W_{2}=\cos u_{1}\dfrac{\partial }{\partial y_{1}}+\dfrac{\partial }{\partial
			x_{2}}-\dfrac{\partial }{\partial y_{2}}+\sin u_{1}\dfrac{\partial }{%
			\partial y_{3}},
	\end{equation*}%
	\begin{equation*}
		W_{3}=-u_{3}\sin u_{1}\dfrac{\partial }{\partial x_{1}}-u_{2}\sin u_{1}%
		\dfrac{\partial }{\partial y_{1}}+u_{3}\cos u_{1}\dfrac{\partial }{\partial
			x_{3}}+u_{2}\cos u_{1}\dfrac{\partial }{\partial y_{3}}
	\end{equation*}%
	and
	\begin{equation*}
		\begin{array}{cc}
			W_{4}=\dfrac{\partial }{\partial t_{1}},~ & W_{5}=\dfrac{\partial }{\partial
				t_{2}}.%
		\end{array}%
	\end{equation*}%
	Then we find
	\begin{equation*}
		\varphi W_{1}=-\cos u_{1}\dfrac{\partial }{\partial y_{1}}-\dfrac{\partial }{%
			\partial y_{2}}+\dfrac{\partial }{\partial x_{2}}-\sin u_{1}\dfrac{\partial
		}{\partial y_{3}},
	\end{equation*}%
	\begin{equation*}
		\varphi W_{2}=\cos u_{1}\dfrac{\partial }{\partial x_{1}}-\dfrac{\partial }{%
			\partial y_{2}}-\dfrac{\partial }{\partial x_{2}}+\sin u_{1}\dfrac{\partial
		}{\partial x_{3}},
	\end{equation*}%
	\begin{equation*}
		\varphi W_{3}=u_{3}\sin u_{1}\dfrac{\partial }{\partial y_{1}}-u_{2}\sin
		u_{1}\dfrac{\partial }{\partial x_{1}}-u_{3}\cos u_{1}\dfrac{\partial }{%
			\partial y_{3}}+u_{2}\cos u_{1}\dfrac{\partial }{\partial x_{3}}
	\end{equation*}%
	and
	\begin{equation*}
		\begin{array}{cc}
			\varphi W_{4}=0,~ & \varphi W_{5}=0.%
		\end{array}%
	\end{equation*}%
	It is easy to see that $\varphi W_{3}$ is orthogonal to$~TM.$ Thus, the
	anti-invariant distribution $D_{1}=span\left\{ W_{3}\right\} $ and $%
	D_{2}=span\left\{ W_{1},~W_{2}\right\} $ is a proper slant distribution with
	slant angle $\theta =arc\cos \left( \dfrac{1}{3}\right) $ such that $\xi
	_{1}=\dfrac{\partial }{\partial t_{1}}$ and $\xi _{2}=\dfrac{\partial }{%
		\partial t_{2}}$ is tangent to $M$. Thus M is a proper pseudo-slant
	submanifold. Also, both the distributions are integrable. When we denote the
	integral manifolds of $D^{\bot }$ and $D^{\theta }$ by $M_{\bot }$ and $%
	M_{\theta },$ respectively. Then the metric tensor $g$ of $M$ is calculated
	as in the following
	\begin{equation*}
		g=3\left( du_{3}^{2}+du_{2}^{2}\right) +dt_{1}^{2}+dt_{2}^{2}+\left(
		u_{1}^{2}+u_{2}^{2}\right) du_{1}^{2}.
	\end{equation*}%
	Thus $M$ is a warped product pseudo-slant submanifold of the form $%
	M=M_{\theta }\times _{f}M_{\bot }$ with the warping function $f=\sqrt{%
		u_{1}^{2}+u_{2}^{2}}.$
\end{example}

There is another example of warped product pseudo-slant submanifold in a nearly Kenmotsu $f-$manifold as follows
\begin{example}
Now let $M$ be a submanifold of $\overline{M}$ given by the following
immersion $\chi $
\begin{equation*}
	\chi \left( u,~v,~w,~t_{1},~t_{2}\right) =\left( 0,~w^{2}\sin
	u,~0,~v^{2}\sin u,~0,~0,~0,~w^{2}\text{+}v^{2},~0,~w^{2}\cos u,~v^{2}\cos
	u,~0,~t_{1},~t_{2}\right) .
\end{equation*}%
Then the tangent space of $M$ spanned by
\begin{equation*}
	X=\cos u\dfrac{\partial }{\partial x_{1}}+\dfrac{\partial }{\partial x_{2}}+%
	\dfrac{\partial }{\partial y_{2}}+\sin u\dfrac{\partial }{\partial x_{3}}+%
	\dfrac{\partial }{\partial x_{4}}+2\dfrac{\partial }{\partial x_{6}},
\end{equation*}%
\begin{equation*}
	Y=\cos u\dfrac{\partial }{\partial y_{1}}+\dfrac{\partial }{\partial x_{2}}-%
	\dfrac{\partial }{\partial y_{2}}+\sin u\dfrac{\partial }{\partial y_{3}}+%
	\dfrac{\partial }{\partial x_{5}}+\dfrac{\partial }{\partial y_{6}},
\end{equation*}%
\begin{equation*}
	Z=-w\sin u\dfrac{\partial }{\partial x_{1}}-v\sin u\dfrac{\partial }{%
		\partial y_{1}}+w\cos u\dfrac{\partial }{\partial x_{3}}+v\cos u\dfrac{%
		\partial }{\partial y_{3}}+3\dfrac{\partial }{\partial y_{4}}+2\dfrac{%
		\partial }{\partial y_{5}}
\end{equation*}%
and
\begin{equation*}
	\begin{array}{cc}
		U=\dfrac{\partial }{\partial t_{1}},~ & V=\dfrac{\partial }{\partial t_{2}}.%
	\end{array}%
\end{equation*}%
Then we find
\begin{equation*}
	\varphi X=\cos u\dfrac{\partial }{\partial y_{1}}+\dfrac{\partial }{\partial
		y_{2}}-\dfrac{\partial }{\partial x_{2}}+\sin u\dfrac{\partial }{\partial
		y_{3}}+\dfrac{\partial }{\partial y_{4}}+2\dfrac{\partial }{\partial y_{6}},
\end{equation*}%
\begin{equation*}
	\varphi Y=-\cos u\dfrac{\partial }{\partial x_{1}}+\dfrac{\partial }{%
		\partial y_{2}}+\dfrac{\partial }{\partial x_{2}}-\sin u\dfrac{\partial }{%
		\partial x_{3}}+\dfrac{\partial }{\partial y_{5}}-\dfrac{\partial }{\partial
		x_{6}},
\end{equation*}%
\begin{equation*}
	\varphi Z=-w\sin u\dfrac{\partial }{\partial y_{1}}+v\sin u\dfrac{\partial }{%
		\partial x_{1}}+w\cos u\dfrac{\partial }{\partial y_{3}}-v\cos u\dfrac{%
		\partial }{\partial x_{3}}-3\dfrac{\partial }{\partial y_{4}}-2\dfrac{%
		\partial }{\partial y_{5}}
\end{equation*}%
and
\begin{equation*}
	\begin{array}{cc}
		\varphi U=0,~ & \varphi V=0.%
	\end{array}%
\end{equation*}%
It is easy to see that $\varphi Z$ is orthogonal to$~TM.$ Thus, the
anti-invariant distribution $D_{1}=span\left\{ Z\right\} $ and $%
D_{2}=span\left\{ X,~Y\right\} $ is a proper slant distribution with slant
angle $\theta =arc\cos \left( \dfrac{\sqrt{10}}{20}\right) $ such that $\xi
_{1}=\dfrac{\partial }{\partial t_{1}}$ and $\xi _{2}=\dfrac{\partial }{%
	\partial t_{2}}$ is tangent to $M$. Thus M is a proper pseudo-slant
submanifold. Also, both the distributions are integrable. When we denote the
integral manifolds of $D^{\bot }$ and $D^{\theta }$ by $M_{\bot }$ and $%
M_{\theta },$ respectively. Then the metric tensor $g$ of $M$ is calculated
as in the following
\begin{equation*}
	g=8\left( w^{2}dw^{2}+v^{2}dv^{2}\right) +dt_{1}^{2}+dt_{2}^{2}+\left(
	w^{4}+v^{4}\right) du^{2}.
\end{equation*}%
Thus $M$ is a warped product pseudo-slant submanifold of the form $%
M=M_{\theta }\times _{f}M_{\bot }$ with the warping function $f=\sqrt{\left(
	w^{4}+v^{4}\right) }.$
\end{example}

Now, we prove some lemmas for the next section. We begin with the following.

\begin{lemma}\label{l42}
	Let $M=M_{\theta }\times _{f}M_{\bot }$ be a non-trivial warped product
	pseudo slant submanifold of a nearly Kenmotsu $f$-manifold $\overline{M}.$
	Then we have
	
	\begin{align}
	g\left( h\left( Z,~Z\right) ,~NTX\right)&=g\left( h\left(
	Z,~TX\right) ,~\varphi Z\right)\notag\\
	& +\left\{ s\sum_{k=1}^{s}\eta ^{k}\left(
	X\right) -\left( X\ln f\right) \right\} \cos ^{2}\theta \left\Vert
	Z\right\Vert ^{2},\\
	g\left( h\left( Z,~Z\right) ,~NX\right)&=g\left( h\left(
	Z,~X\right) ,~\varphi Z\right) -\left( TX\ln f\right) \left\Vert
	Z\right\Vert ^{2},
	\end{align}
	for any $X\ $on $M_{\theta }$ and $Z~$on $M_{\bot }$, where the structure
	vector fields $\xi _{k}^{\prime }s$ are tangent to $M_{\theta }.$
\end{lemma}

\begin{proof}
	By using (\ref{a}) and (\ref{f}), then we get
	\begin{equation*}
		g\left( h\left( Z,~Z\right) ,~NTX\right) =g\left( \overline{\nabla }%
		_{Z}Z,~NTX\right) =g\left( \overline{\nabla }_{Z}Z,~\varphi TX\right)
		-g\left( \overline{\nabla }_{Z}Z,~T^{2}X\right) .
	\end{equation*}%
	From (\ref{30}), it follows
	\begin{eqnarray*}
		g\left( h\left( Z,~Z\right) ,~NTX\right) &=&-g\left( \varphi \overline{%
			\nabla }_{Z}Z,~TX\right) \\
		&&+\cos ^{2}\theta \left\{ g\left( \overline{\nabla }_{Z}Z,~X\right)
		-\sum_{k=1}^{s}\eta ^{k}\left( X\right) g\left( \overline{\nabla }_{Z}Z,~\xi
		_{k}\right) \right\} .
	\end{eqnarray*}%
	By virtue of properties, we deduce
	\begin{eqnarray*}
		g\left( h\left( Z,~Z\right) ,~NTX\right) &=&g\left( \left( \overline{\nabla }%
		_{Z}\varphi \right) Z,~TX\right) -g\left( \overline{\nabla }_{Z}\varphi
		Z,~TX\right) \\
		&&+\cos ^{2}\theta g\left( \nabla _{Z}X,~Z\right) +\cos ^{2}\theta
		\sum_{k=1}^{s}\eta ^{k}\left( X\right) g\left( \overline{\nabla }_{Z}\xi
		_{k},~Z\right) .
	\end{eqnarray*}%
	Taking into account of $\sum_{k=1}^{s}\left( \xi _{k}\ln f\right) =s$ and
	using Lemma 1 (ii), we get $\left( i\right) .$ The second property of lemma
	can be easily gotten by interchanging $X$ by $TX$ in $\left( ii\right) $ of
	this lemma. This completes proof of lemma
\end{proof}
 For the second case of warped product pseudo-slant submanifold of type $M_{\bot }\times _{f}M_{\theta }$ where structure vector fields $\xi_s$ are tangent to base manifold. We derive some important lemmas which will use in our main results.
\begin{lemma}\label{l43}
Let $M=M_{\bot }\times _{f}M_{\theta }$ be a warped product pseudo slant
submanifold of a nearly Kenmotsu $f$-manifold $\overline{M}$ such that the
structure vector fields $\xi _{k}^{\prime }s$~are tangent to $M_{\bot },$
for $1\leq k\leq s.$ Then we have
\begin{equation*}
g\left( h\left( X,~TX\right) ,~\varphi Z\right) =g\left( h\left( X,~Z\right)
,~NTX\right) +\frac{1}{3}\left\{ \sum_{k=1}^{s}\eta ^{k}\left( Z\right)
-\left( Z\ln f\right) \right\} \cos ^{2}\theta \left\Vert X\right\Vert ^{2},
\end{equation*}%
for any vector field $X~$on $M_{\theta }$ and $Z$ on $M_{\bot }.$
\end{lemma}

\begin{proof}
Let $M=M_{\bot }\times _{f}M_{\theta }$ be a warped product pseudo slant
submanifold of a nearly Kenmotsu $f$-manifold $\overline{M}.$ By using (\ref%
{a}), we find
\begin{equation*}
g\left( h\left( X,~TX\right) ,~\varphi Z\right) =-g\left( \varphi \overline{%
\nabla }_{X}TX,~Z\right) .
\end{equation*}%
By the virtue of the covariant derivative of $\varphi ,$we get
\begin{equation*}
g\left( h\left( X,~TX\right) ,~\varphi Z\right) =g\left( \left( \overline{%
\nabla }_{X}\varphi \right) TX,~Z\right) -g\left( \overline{\nabla }%
_{X}\varphi TX,~Z\right) .
\end{equation*}%
From (\ref{11}) and Theorem 1, we obtain
\begin{equation*}
g\left( h\left( X,~TX\right) ,~\varphi Z\right) =-g\left( \left( \overline{%
\nabla }_{TX}\varphi \right) X,~Z\right) -\cos ^{2}\theta g\left( \overline{%
\nabla }_{X}Z,~X\right) +g\left( h\left( X,~Z\right) ,~NTX\right) .
\end{equation*}%
Hence by using Lemma \ref{l41} (ii) and the covariant derivative of $\varphi $, we
derive
\begin{align*}
g\left( h\left( X,~TX\right) ,~\varphi Z\right)=&-g\left( \overline{\nabla }%
_{TX}\varphi X,~Z\right) -g\left( \overline{\nabla }_{TX}X,~\varphi Z\right)
\\
&+g\left( h\left( X,~Z\right) ,~NTX\right) -\cos ^{2}\theta \left( Z\ln
f\right) \left\Vert X\right\Vert ^{2}.%
\end{align*}%
Taking into account of (\ref{b}) and (\ref{f}), we deduce
\begin{align*}
2g\left( h\left( X,~TX\right) ,~\varphi Z\right) =&g\left( \nabla
_{TX}Z,~TX\right) -g\left( \overline{\nabla }_{TX}NX,~Z\right) \\
&+g\left( h\left( X,~Z\right) ,~NTX\right) -\cos ^{2}\theta \left( Z\ln
f\right) \left\Vert X\right\Vert ^{2}.%
\end{align*}%
Now, by using Lemma \ref{l41} (ii) and (\ref{31}), it follows that
\begin{equation}
2g\left( h\left( X,~TX\right) ,~\varphi Z\right) =g\left( h\left(
TX,~Z\right) ,~NX\right) +g\left( h\left( X,~Z\right) ,~NTX\right) .
\label{35}
\end{equation}%
On the other hand, for any $X~$on $M_{\theta }$ and $Z~$on $M_{\bot },$ we
conclude that%
\begin{equation*}
g\left( h\left( X,~Z\right) ,~NTX\right) =g\left( \overline{\nabla }%
_{Z}X,~NTX\right) .
\end{equation*}%
Since $\xi _{k}~$is tangent to $TM_{\bot }~$for each $1\leq k\leq s,$ (\ref%
{a}) and (\ref{f}) imply
\begin{equation*}
g\left( h\left( X,~Z\right) ,~NTX\right) =-g\left( \varphi \overline{\nabla }%
_{Z}X,~TX\right) +\cos ^{2}g\left( \nabla _{Z}X,~X\right) .
\end{equation*}%
Again by using Lemma \ref{l41} (ii) and the covariant derivative of $\varphi $, it
is said that
\begin{align*}
g\left( h\left( X,~Z\right) ,~NTX\right)=&g\left( \left( \overline{\nabla }%
_{Z}\varphi \right) X,~TX\right) -g\left( \overline{\nabla }_{Z}\varphi
X,~TX\right) \\
&+\cos ^{2}\theta \left( Z\ln f\right) \left\Vert X\right\Vert
^{2}.
\end{align*}%
By the virtue of (\ref{11}) and (\ref{31}), the last equation takes the form
\begin{align*}
g\left( h\left( X,~Z\right) ,~NTX\right)=&-g\left( \left( \overline{\nabla }%
_{X}\varphi \right) Z,~TX\right) -\sum_{k=1}^{s}\eta ^{k}\left( Z\right)
\cos ^{2}\theta \left\Vert X\right\Vert ^{2} \\
&-g\left( \nabla _{Z}TX,~TX\right) -g\left( \overline{\nabla }%
_{Z}NX,~TX\right)\\
& +\cos ^{2}\theta \left( X\ln f\right) \left\Vert
Z\right\Vert ^{2}.%
\end{align*}%
By using Lemma 1 (ii), (\ref{b}) and (\ref{31}), it follows that
\begin{align*}
g\left( h\left( X,~Z\right) ,~NTX\right)=&-g\left( \overline{\nabla }%
_{X}\varphi Z,~TX\right) -g\left( \overline{\nabla }_{X}Z,~\varphi TX\right)
\\
&-\sum_{k=1}^{s}\eta ^{k}\left( Z\right) \cos ^{2}\theta \left\Vert
X\right\Vert ^{2}+g\left( h\left( Z,~TX\right) ,~NX\right) .%
\end{align*}%
From Lemma 1 (ii) and (\ref{a}), we have
\begin{eqnarray}
g\left( h\left( X,~TX\right) ,~\varphi Z\right) &=&2g\left( h\left(
X,~Z\right) ,~NTX\right)  \notag \\
&&-\left\{ \left( Z\ln f\right) -\sum_{k=1}^{s}\eta ^{k}\left( Z\right)
\right\} \cos ^{2}\theta \left\Vert X\right\Vert ^{2}  \label{36} \\
&&-g\left( h\left( Z,~TX\right) ,~NX\right) .  \notag
\end{eqnarray}%
Hence from (\ref{35}) and (\ref{36}), we obtain
\begin{equation*}
g\left( h\left( X,~TX\right) ,~\varphi Z\right) -g\left( h\left( X,~Z\right)
,~NTX\right) =\frac{1}{3}\left\{ \sum_{k=1}^{s}\eta ^{k}\left( Z\right)
-\left( Z\ln f\right) \right\} \cos ^{2}\theta \left\Vert X\right\Vert ^{2},
\end{equation*}%
which gives us the desired result.
\end{proof}

As a consequence of this lemma, we can give the following corollary.

\begin{corollary}\label{coro41}
Let $M=M_{\bot }\times _{f}M_{\theta }$ be a totally geodesic warped product
pseudo slant submanifold of a nearly Kenmotsu $f$-manifold $\overline{M}.$
Then at least on the following statements is true:

$\left( i\right) ~M$ is an anti-invariant submanifold,

$\left( ii\right) ~M$ is an invariant submanifold or,

$\left( iii\right) ~\left( Z\ln f\right) =\sum_{k=1}^{s}\eta ^{k}\left(
Z\right) .$
\end{corollary}

\begin{lemma}\label{l44}
Let $M=M_{\bot }\times _{f}M_{\theta }$ be a warped product pseudo slant
submanifold of a nearly Kenmotsu $f$-manifold $\overline{M}.$ Then the
followings hold.

$\left( i\right) ~g\left( h\left( X,~X\right) ,~\varphi Z\right) =g\left(
h\left( Z,~X\right) ,~NX\right) ,$

$\left( ii\right) ~g\left( h\left( TX,~TX\right) ,~\varphi Z\right) =g\left(
h\left( Z,~TX\right) ,~NTX\right) ,$

for any $X$~on $M_{\theta }$ and $Z~$on $M_{\bot }.$
\end{lemma}

\begin{proof}
By using (\ref{a}), we have
\begin{equation*}
g\left( h\left( X,~X\right) ,~\varphi Z\right) =g\left( \overline{\nabla }%
_{X}X,~\varphi Z\right) =-g\left( \varphi \overline{\nabla }_{X}X,~Z\right) .
\end{equation*}%
From the properties we obtain
\begin{equation*}
g\left( h\left( X,~X\right) ,~\varphi Z\right) =g\left( \left( \overline{%
\nabla }_{X}\varphi \right) X,~Z\right) -g\left( \overline{\nabla }%
_{X}\varphi X,~Z\right) .
\end{equation*}%
By virtue of (\ref{0}) and (\ref{11}), then we derive
\begin{equation*}
g\left( h\left( X,~X\right) ,~\varphi Z\right) =g\left( TX,~\overline{\nabla
}_{X}Z\right) -g\left( \overline{\nabla }_{X}NX,~Z\right) .
\end{equation*}%
Taking into account of Lemma \ref{l41} (ii) and (\ref{a}), it follows that
\begin{equation*}
g\left( h\left( X,~X\right) ,~\varphi Z\right) =\left( Z\ln f\right) g\left(
TX,~X\right) +g\left( A_{NX}X,~Z\right) .
\end{equation*}%
Since $X$ and $TX$ are orthogonal vector fields and in view of (\ref{c}), we
conclude that
\begin{equation}
g\left( h\left( X,~X\right) ,~\varphi Z\right) =g\left( h\left( X,~Z\right)
,~NX\right) ,  \label{37}
\end{equation}%
which gives us $\left( i\right) .$ By interchanging $X$ by $TX$ in (\ref{37}%
), we have the last result of this lemma.
\end{proof}

\section{Inequality for a Warped Product Pseudo Slant Submanifold of the
form $M_{\bot }\times _{f}M_{\protect\theta }$}

In this section, we obtain a geometric inequality of warped product pseudo
slant submanifold in terms of the second fundamental form such that $\xi
_{k} $ is tangent to the anti-invariant submanifold and the mixed totally
geodesic submanifold for each $1\leq k\leq s.$

Now let $M=M_{\bot }\times _{f}M_{\theta }$ be $m$-dimensional warped
product pseudo slant submanifold of $\left( 2m+s\right) $-dimensional nearly
Kenmotsu $f$-manifold $\overline{M}$ with $M_{\theta }$ of dimension $%
d_{1}=2\beta $ and $M_{\bot }~$of dimension $d_{2}=\alpha +s,$ where $%
M_{\theta }$ and $M_{\bot }$ are the integral manifolds of $D_{\theta }$ and
$D^{\bot }$, respectively. Then we consider $\{e_{1},~\ldots e_{\alpha
},~e_{d_{2}=\alpha +1}=\xi _{1},~\ldots ,~e_{d_{2}=\alpha +s}=\xi _{s}\}$
and $\{e_{\alpha +s+1}=e_{1}^{\ast },~\ldots e_{\alpha +s+\beta +1}=e_{\beta
}^{\ast },~e_{\alpha +s+\beta +2}=e_{\beta +1}^{\ast }=\sec \theta
Te_{1}^{\ast },~\ldots ,~e_{\alpha +s+2\beta }=e_{2\beta }^{\ast }=\sec
\theta Te_{2\beta }^{\ast }\}$ are orthonormal basis of $D^{\bot }$and $%
D_{\theta },$ respectively. Hence the orthonormal basis of the normal
subbundles $\varphi D^{\bot },~ND_{\theta }$ and $\mu $ are $\{e_{q+1}=%
\overline{e}_{1}=\varphi e_{1},~\ldots ,~e_{q+\alpha }=\overline{e}_{\alpha
}=\varphi e_{\alpha }\},~\{e_{q+\alpha +1}=\overline{e}_{\alpha +1}=%
\widetilde{e}_{1}=\csc \theta Ne_{1}^{\ast },~\ldots ,~e_{q+\alpha +\beta }=%
\overline{e}_{\alpha +\beta }=\widetilde{e}_{\beta }=\csc \theta Ne_{\beta
}^{\ast },~e_{q+\alpha +\beta +1}=\overline{e}_{\alpha +\beta +1}=\widetilde{%
e}_{\beta +1}=\csc \theta NTe_{1}^{\ast },~\ldots e_{m+\alpha +2\beta }=%
\overline{e}_{\alpha +2\beta }=\widetilde{e}_{2\beta }=\csc \theta
NTe_{\beta }^{\ast }\}~\{e_{2m-1}=\overline{e}_{m},~\ldots ,~e_{2n+s}=%
\overline{e}_{2(n-m+s)}\}$, respectively.

\begin{theorem}\label{th51}
Let $M=M_{\bot }\times _{f}M_{\theta }$ be $q$-dimensional mixed totally
geodesic warped product pseudo slant submanifold of a $\left( 2m+s\right) $%
-dimensional nearly Kenmotsu $f$-manifold $\overline{M}~$such that $\xi
_{k}\in \Gamma \left( TM_{\bot }\right) $, where $M_{\bot }$ is an
anti-invariant submanifold of dimension $d_{2}$, $M_{\theta }$ is a proper
slant submanifold of dimension $d_{1}$ of $\overline{M}$ and $\nabla ^{\bot
} $ denotes the $\nabla ^{\bot }$ is normal connection with respect to
normal subbunle. Then we have

$\left( i\right) ~$The squared norm of the second fundamental form of $M$ is
given by
\begin{equation}
\left\Vert h\right\Vert ^{2}\geq \frac{2\beta }{9}\cos ^{2}\theta \left\{
\left\Vert \nabla ^{\bot }\ln f\right\Vert ^{2}-s^{2}\right\} .  \label{38}
\end{equation}

$\left( ii\right) ~$The equality case holds in (\ref{38}), if $M_{\bot }$ is
totally geodesic and $M_{\theta }$ is a totally umbilical submanifold into $%
\overline{M}.$
\end{theorem}

\begin{proof}
The squared norm of the second fundamental form is defined by
\begin{equation*}
\left\Vert h\right\Vert ^{2}=\left\Vert h\left( D_{\theta },~D_{\theta
}\right) \right\Vert ^{2}+\left\Vert h\left( D^{\bot },~D^{\bot }\right)
\right\Vert ^{2}+2\left\Vert h\left( D_{\theta },~D^{\bot }\right)
\right\Vert ^{2}.
\end{equation*}%
Since $M$ is mixed totally geodesic we obtain
\begin{equation}
\left\Vert h\right\Vert ^{2}=\left\Vert h\left( D_{\theta },~D_{\theta
}\right) \right\Vert ^{2}+\left\Vert h\left( D^{\bot },~D^{\bot }\right)
\right\Vert ^{2}.  \label{39}
\end{equation}%
From (\ref{33}), we have
\begin{align*}
\left\Vert h\right\Vert ^{2}\geq \sum_{l=q+1}^{2m+s}\sum_{i,~j=1}^{2\beta
}g\left( h\left( e_{i}^{\ast },~e_{j}^{\ast }\right) ,~e_{l}\right) ^{2}.
\end{align*}%
By rewriting the last equation as in the components of $\varphi D^{\bot
},~ND_{\theta }~$and $v$, then we get
\begin{align}
\left\Vert h\right\Vert ^{2}&\geq \sum_{l=1}^{\alpha
}\sum_{i,~j=1}^{2\beta }g\left( h\left( e_{i}^{\ast },~e_{j}^{\ast }\right)
,~\overline{e}_{l}\right) ^{2}+\sum_{l=\alpha +1}^{2\beta +\alpha
}\sum_{i,~j=1}^{2\beta }g\left( h\left( e_{i}^{\ast },~e_{j}^{\ast }\right)
,~\overline{e}_{l}\right) ^{2}  \label{40} \\
&+\sum_{l=q}^{2\left( m-q+s\right) }\sum_{i,~j=1}^{2\beta }g\left( h\left(
e_{i}^{\ast },~e_{j}^{\ast }\right) ,~\overline{e}_{l}\right) ^{2},  \notag
\end{align}%
which implies
\begin{equation*}
\left\Vert h\right\Vert ^{2}\geq \sum_{l=1}^{\alpha }\sum_{i,~j=1}^{2\beta
}g\left( h\left( e_{i}^{\ast },~e_{j}^{\ast }\right) ,~\overline{e}%
_{l}\right) ^{2}.
\end{equation*}%
By considering another adapted frame for $D^{\theta },~$we derive
\begin{align*}
\left\Vert h\right\Vert ^{2} &\geq\sum_{i=1}^{\alpha }\sum_{r,~k=1}^{\beta
}g\left( h\left( e_{r}^{\ast },~e_{k}^{\ast }\right) ,~\overline{e}%
_{i}\right) ^{2}+\sec ^{2}\theta \sum_{i=1}^{\alpha }\sum_{r,~k=1}^{\beta
}g\left( h\left( Te_{r}^{\ast },~e_{k}^{\ast }\right) ,~\overline{e}%
_{i}\right) ^{2} \\
&+\sec ^{2}\theta \sum_{i=1}^{\alpha }\sum_{r,~k=1}^{\beta }g\left( h\left(
e_{r}^{\ast },~Te_{k}^{\ast }\right) ,~\overline{e}_{i}\right) ^{2}+\sec
^{4}\theta \sum_{i=1}^{p}\sum_{r,~k=1}^{\beta }g\left( h\left( Te_{r}^{\ast
},~Te_{k}^{\ast }\right) ,~\overline{e}_{i}\right) ^{2}.
\end{align*}%
For a mixed totally geodesic submanifold, since the first and last terms of
the right hand side in the above equation vanish identically by using Lemma
\ref{l44}, then we obtain
\begin{equation*}
\left\Vert h\right\Vert ^{2}\geq 2\sec ^{2}\theta \sum_{i=1}^{\alpha
}\sum_{r,~k=1}^{\beta }g\left( h\left( Te_{r}^{\ast },~e_{k}^{\ast }\right)
,~\overline{e}_{i}\right) ^{2}.
\end{equation*}%
Hence from Lemma \ref{l43}, for a mixed totally geodesic submanifold and by
considering the fact that $\eta ^{u}\left( e_{i}\right) =0,$ for each $%
i=1,~2,~\ldots d_{2}-1$ and $1\leq u\leq s$ for an orthonormal frame, it
follows that
\begin{equation}
\left\Vert h\right\Vert ^{2}\geq \frac{2}{9}\cos ^{2}\theta
\sum_{i=1}^{\alpha }\sum_{r=1}^{\beta }\left( \overline{e}_{i}\ln f\right)
^{2}g\left( e_{r}^{\ast },~e_{r}^{\ast }\right) ^{2}.  \label{41}
\end{equation}%
By adding and subtracting the same term $\sum_{u=1}^{s}\xi _{u}\ln f$ in (%
\ref{41}), it implies that
\begin{align*}
\left\Vert h\right\Vert ^{2}&\geq\frac{2}{9}\cos ^{2}\theta
\sum_{i=1}^{\alpha +1}\sum_{r=1}^{\beta }\left( \overline{e}_{i}\ln f\right)
^{2}g\left( e_{r}^{\ast },~e_{r}^{\ast }\right) ^{2}\\
&-\frac{2}{9}\cos
^{2}\theta \sum_{r=1}^{\beta }\left( \sum_{u=1}^{s}\xi _{u}\ln f\right)
^{2}g\left( e_{r}^{\ast },~e_{r}^{\ast }\right) ^{2}.
\end{align*}%
We can easily get $\sum_{u=1}^{s}\xi _{u}\ln f=s$ similar to previous
studies for a warped product submanifold of a nearly Kenmotsu $f$-manifold.
From the last equation, it is said that
\begin{equation*}
\left\Vert h\right\Vert ^{2}\geq \frac{2\beta }{9}\cos ^{2}\theta \left\{
\left\Vert \nabla ^{\bot }\ln f\right\Vert ^{2}-s^{2}\right\} .
\end{equation*}%
If the equality case holds in the above equation, then from the terms left
in (\ref{39}), we arrive at
\begin{equation*}
h\left( D^{\bot },~D^{\bot }\right) =0,
\end{equation*}%
which implies that $M_{\bot }$ is totally geodesic in $\overline{M}.$ In a
similar way, from the second and third terms in (\ref{40}), we deduce
\begin{equation*}
g\left( h\left( D_{\theta },~D_{\theta }\right) ,~ND_{\theta }\right)
=0~~~~~~g\left( h\left( D_{\theta },~D_{\theta }\right) ,~v\right) =0,
\end{equation*}%
which means
\begin{equation*}
\begin{array}{cccc}
h\left( D_{\theta },~D_{\theta }\right) \bot ND_{\theta },~ & h\left(
D_{\theta },~D_{\theta }\right) \bot v & \Longrightarrow & h\left( D_{\theta
},~D_{\theta }\right) \subset \varphi D^{\bot }.%
\end{array}%
\end{equation*}%
This completes the proof.
\end{proof}

\begin{rem}
	For globally frame $f-$manifold. If we substitute $s=1$ in Theorem \ref{th51}. Then nearly Kenmotsu $f-$manifold become nearly Kenmotsu manifold. Therefore Theorem \ref{th51} coincide with Theorem 4.1 in \cite{Ali}. This means that Theorem \ref{th51} generalize Theorem 4.1 from \cite{Ali}.
\end{rem}

\section{Inequality for a Warped Product Pseudo Slant Submanifold of the
form $M_{\protect\theta }\times _{f}M_{\bot }$}

In this part, we obtain a geometric inequality of warped product pseudo
slant submanifold in terms of the second fundamental form such that $\xi
_{k} $ is tangent to the slant submanifold for each $1\leq k\leq s.$ By
assuming $\xi _{k}$ is tangent to $D_{\theta }$, then we can use the last
frame.

\begin{theorem}\label{th61}
Let $M=M_{\theta }\times _{f}M_{\bot }$ be a $q$-dimensional mixed totally
geodesic warped product pseudo slant submanifold of a $\left( 2m+s\right) $%
-dimensional nearly Kenmotsu $f$-manifold $\overline{M}~$such that $\xi
_{k}\in \Gamma \left( TM_{\theta }\right) $, where $M_{\bot }$ is an
anti-invariant submanifold of dimension $d_{2}=\alpha $ and $M_{\theta }$ is
a proper slant submanifold of dimension $d_{1}=2\beta +s$ of $\overline{M}.$
Then we have

$\left( i\right) ~$The squared norm of the second fundamental form of $M$ is
given by
\begin{equation}
\left\Vert h\right\Vert ^{2}\geq \alpha \left\{ \cot^{2}\theta \left(
\left\Vert \nabla ^{\theta }\ln f\right\Vert ^{2}-s^{2}\right)
\right\} .
\label{42}
\end{equation}

$\left( ii\right) ~$The equality case holds in (\ref{42}), if $M_{\theta }$
is totally geodesic and $M_{\bot }$ is a totally umbilical submanifold of $%
\overline{M}.$
\end{theorem}

\begin{proof}
By virtue of definition of second fundamental form, we have

\begin{equation*}
\left\Vert h\right\Vert ^{2}=\left\Vert h\left( D_{\theta },~D_{\theta
}\right) \right\Vert ^{2}+\left\Vert h\left( D^{\bot },~D^{\bot }\right)
\right\Vert ^{2}+2\left\Vert h\left( D_{\theta },~D^{\bot }\right)
\right\Vert ^{2}.
\end{equation*}%
Since $M$ is mixed totally geodesic we derive
\begin{equation}
\left\Vert h\right\Vert ^{2}=\left\Vert h\left( D_{\theta },~D_{\theta
}\right) \right\Vert ^{2}+\left\Vert h\left( D^{\bot },~D^{\bot }\right)
\right\Vert ^{2}.  \label{43}
\end{equation}%
By using (\ref{33}), then we obtain
\begin{equation*}
\left\Vert h\right\Vert ^{2}\geq \sum_{l=q+1}^{2m+s}\sum_{r,~k=1}^{\alpha
}g\left( h\left( e_{r},~e_{k}\right) ,~e_{l}\right) ^{2}.
\end{equation*}%
The above equation can be written in the components of $\varphi D^{\bot
},~ND_{\theta }~$and $v$ as
\begin{align}
\left\Vert h\right\Vert ^{2} \geq&\sum_{l,~r,~k=1}^{\alpha }g\left(
h\left( e_{r},~e_{k}\right) ,~\overline{e}_{l}\right) ^{2}+\sum_{l=\alpha
+1}^{2\beta +\alpha }\sum_{r,~k=1}^{\alpha }g\left( h\left(
e_{r},~e_{k}\right) ,~\overline{e}_{l}\right) ^{2}  \label{44} \\
&+\sum_{l=q}^{2\left( m-q+s\right) }\sum_{r,~k=1}^{\alpha }g\left( h\left(
e_{r},~e_{k}\right) ,~\overline{e}_{l}\right) ^{2},  \notag
\end{align}%
which gives us
\begin{equation*}
\left\Vert h\right\Vert ^{2}\geq \sum_{l=1}^{2\beta }\sum_{r,~k=1}^{\alpha
}g\left( h\left( e_{r},~e_{k}\right) ,~\overline{e}_{l}\right) ^{2}.
\end{equation*}%
By taking into account of another adapted frame for $ND^{\theta },~$we get
\begin{align*}
\left\Vert h\right\Vert ^{2}\geq&\csc ^{2}\theta \sum_{j=1}^{\beta
}\sum_{r=1}^{\alpha }g\left( h\left( e_{r},~e_{r}\right) ,~Ne_{j}^{\ast
}\right) ^{2}\\
&+\csc ^{2}\theta \sec ^{2}\theta \sum_{j=1}^{\beta
}\sum_{r=1}^{\alpha }g\left( h\left( e_{r},~e_{r}\right) ,~NTe_{j}^{\ast
}\right) ^{2}.
\end{align*}%
Thus from Lemma \ref{l42}, for a mixed totally geodesic submanifold and by
considering the fact that $\eta ^{u}\left( e_{j}\right) =0,$ for each $%
j=1,~2,~\ldots d_{1}-1$ and $1\leq u\leq s$ for an orthonormal frame, it
implies that
\begin{align*}
\left\Vert h\right\Vert ^{2}&\geq \csc ^{2}\theta \sum_{j=1}^{\beta
}\sum_{r=1}^{\alpha }\left( Te_{j}^{\ast }\ln f\right) ^{2}g\left(
e_{r},~e_{r}\right)\\
& +\csc ^{2}\theta \sec ^{2}\theta \sum_{j=1}^{\beta
}\sum_{r=1}^{\alpha }\left( e_{j}^{\ast }\ln f\right) ^{2}g\left(
e_{r},~e_{r}\right) .
\end{align*}%
From the hypothesis, we deduce
\begin{align*}
\left\Vert h\right\Vert ^{2}&\geq\alpha \csc ^{2}\theta \sum_{j=1}^{\beta
}\left( Te_{j}^{\ast }\ln f\right) ^{2}+\alpha \cot ^{2}\theta
\sum_{j=1}^{\beta }\left( e_{j}^{\ast }\ln f\right) ^{2}\\
&=\alpha\csc^2\theta\sum_{r=1}^{2\beta+s}g\big(e_r^*, T\nabla^\theta\ln f\big)^2-\alpha\csc^2\theta\sum_{r=\beta+1}^{2\beta}g\big(e_r^*, T\nabla^\theta\ln f\big)^2\\
&-\alpha\csc^2\theta\sum_{u=1}^{s}\big(Te_{2\beta+u}^*\ln f\big)^2+\alpha\cot^2\theta\sum_{r=1}^{\beta}\big(e_r^*\ln f\big)^2.
\end{align*}%
As we seen that $Te_{2\beta+u}^*=T\xi_s=0$. Then using \eqref{34}, we derive at
\begin{align*}
\left\Vert h\right\Vert ^{2}\geq&\alpha\csc^2\theta\|T\nabla^\theta\ln f\|^2-\alpha\csc^2\theta\sum_{r=1}^{\beta}g\big(e_{\beta+r}^*, T\nabla^\theta\ln f\big)^2\\
&+\alpha\cot^2\theta\sum_{r=1}^{\beta}\big(e_r^*\ln f\big)^2.
\end{align*}
From virtue of \eqref{31} and the fact that $\sum_{k=1}^{s}\left( \xi _{k}\ln f\right)=s$, we find that
\begin{align*}
\left\Vert h\right\Vert ^{2}&\geq\alpha\cot^2\theta\bigg(\|\nabla^\theta\ln f\|^2-s^2\bigg)-\alpha\csc^2\theta\sum_{r=1}^{\beta}g\big(\sec\theta Te_\beta^*, T\nabla^\theta\ln f\big)^2\\
&+\alpha\cot^2\theta\sum_{r=1}^{\beta}\big(e_r^*\ln f\big)^2.
\end{align*}
In view of the trigonometric functions and from \eqref{31}, we conclude
\begin{align*}
\left\Vert h\right\Vert ^{2}&\geq\alpha\cot^2\theta\bigg(\|\nabla^\theta\ln f\|^2-s^2\bigg)-\alpha\csc^2\theta\sec^2\theta\cos^4\theta\sum_{r=1}^{\beta}(e_r^*\ln f)^2\\
&+\alpha\cot^2\theta\sum_{r=1}^{\beta}\big(e_r^*\ln f\big)^2,
\end{align*}%
 which implies that
\begin{align*}
\left\Vert h\right\Vert ^{2}\geq&\alpha\cot^2\theta\bigg(\|\nabla^\theta\ln f\|^2-s^2\bigg)-\alpha\cot^2\theta\sum_{r=1}^{\beta}(e_r^*\ln f)^2\\
&+\alpha\cot^2\theta\sum_{r=1}^{\beta}\big(e_r^*\ln f\big)^2.
\end{align*}%
Thus the above equation yields
\begin{equation*}
\left\Vert h\right\Vert ^{2}\geq \alpha \left\{ \cot^{2}\theta \left(
\left\Vert \nabla ^{\theta }\ln f\right\Vert ^{2}-s^{2}\right)
\right\} .
\end{equation*}%
If the equality holds, by using the terms left hand side in (\ref{43}) and (%
\ref{44}), we get the following conditions
\begin{equation*}
\begin{array}{ccc}
\left\Vert h\left( D,~D\right) \right\Vert =0, & g\left( h\left( D^{\bot
},~D^{\bot }\right) ,~\varphi D^{\bot }\right) =0, & g\left( h\left( D^{\bot
},~D^{\bot }\right) ,~v\right) =0,%
\end{array}%
\end{equation*}%
where $D=D_{\theta }\oplus \overline{\xi },~\overline{\xi }%
=\sum_{u=1}^{s}\xi _{u}.$ This implies that $M_{\theta }$ is totally
geodesic in $\overline{M}~$and $h\left(D^{\bot },~D^{\bot }\right)
\subseteq ND_{\theta }.~$On the other hand, using Lemma 4 for a mixed
totally geodesic submanifold we get
\begin{equation*}
g\left( h\left( Z,~W\right) ,~NX\right) =-\left( TX\ln f\right) g\left(
Z,~W\right) ,
\end{equation*}%
for all vector fields $Z,~W\ $on$~M_{\bot}$and $X$ on $M_{\theta}.$ The
last equations means that $M_{\bot}$ is a totally umbilical submanifold of $%
\overline{M}~$and so the equality case holds. This completes the proof of theorem
\end{proof}

\section*{Applications}
In this paper, we study warped product pseudo-slant submanifolds of nearly
Kenmotsu $f$-manifolds. We generalize some previous results on nearly K\"{a}%
hler manifolds \cite{Uddin3} and nearly Kenmotsu manifolds \cite{Ali}. That is, if we consider $s=0$ in Theorem \ref{th61}, then by the definition of globally frame $f-$manifold. It leads that the nearly Kenmotsu $f-$maniold turn into nearly Kaehler manifiold. So Theorem \ref{th61} generalize Theorem 4.1 in \cite{Uddin3} for warped product pseudo-slant submanifold of nearly Kaehler manifold. \\

On the other hand, if we choose $s=1$ then we give the following theorem as a consequence of Theorem \ref{th61} follows
\begin{theorem}
Let $M=M_{\theta }\times _{f}M_{\bot }$ be an $q$-dimensional mixed totally
geodesic warped product pseudo slant submanifold of a $\left( 2m+1\right) $%
-dimensional nearly Kenmotsu manifold $\overline{M}~$such that $\xi
\in \Gamma \left( TM_{\theta }\right) $, where $M_{\bot }$ is an
anti-invariant submanifold of dimension $d_{2}=\alpha $ and $M_{\theta }$ is
a proper slant submanifold of dimension $d_{1}=2\beta+1$ of $\overline{M}.$
Then we have

$\left( i\right) ~$The squared norm of the second fundamental form of $M$ is
given by
\begin{equation}
\left\Vert h\right\Vert ^{2}\geq \alpha \left\{ \cot^{2}\theta \left(
\left\Vert \nabla ^{\theta }\ln f\right\Vert ^{2}-1\right)
\right\} .
\label{42a}
\end{equation}

$\left( ii\right) ~$The equality case holds in (\ref{42}), if $M_{\theta }$
is totally geodesic and $M_{\bot }$ is a totally umbilical submanifold of $%
\overline{M}.$
\end{theorem}


The above result agrees and modified version with the inequality for a warped product pseudo-slant sub-manifold of nearly Kenmotsu manifold obtained in \cite{Ali}

\end{document}